\documentclass[a4paper,10.5pt]{amsart}

\usepackage{amsmath}
\usepackage{amssymb}
\usepackage{amsthm}
\usepackage{amscd}
\usepackage[dvips]{graphicx}
\usepackage{enumerate}
\usepackage{breqn}
\usepackage{enumerate}

\newtheorem{theorem}{Theorem}[section]
\newtheorem{lemma}{Lemma}[section]

\newtheorem{prop}{Proposition}[section]

\numberwithin{equation}{section}
\newcommand{\real}{\mathbb{R}}

\def\eqn {\begin{equation}}
\def\eeqn {\end{equation}}
\def\C{{\mathbb C}}
\def\real{{\mathbb R}}

\def\ep{\epsilon}

\def\pa{\partial}

\def\C{\mathcal C}
\def\E{\mathcal E}

\def\F{\mathcal F}
\def\K{\mathcal K}

\def\O{\mathcal O}

\def\ka{\kappa}

\begin{document}
\title{RAPIDLY ROTATING WHITE DWARFS}
\author{Walter A. Strauss} 
\address{Department of Mathematics, Brown University, Providence, RI 02912}
\email{wstrauss@math.brown.edu}
\author{Yilun Wu}
\address{Department of Mathematics, University of Oklahoma, Norman, OK 73019}
\email{allenwu@ou.edu}
\date{}

\begin{abstract}
A rotating star may be modeled as a continuous system of particles attracted to each other by gravity and 
with a given total mass and prescribed angular velocity.   
Mathematically this leads to the Euler-Poisson system.  
A white dwarf star is modeled by a very particular, and rather delicate, equation of state for the pressure as 
a function of the density.  
We prove an existence theorem for rapidly rotating white dwarfs  that 
depend continuously on the speed of rotation.  
The key tool is global continuation theory, combined with a delicate limiting process.  
The solutions form a connected set $\K$ in an appropriate function space. 
As the speed of rotation increases, we prove that either the supports of white dwarfs in $\K$ become unbounded 
or their densities become unbounded.  
We also discuss the polytropic case with the critical exponent $\gamma=4/3$.  

\end{abstract}

\maketitle

\section{Introduction}  
A white dwarf is a very dense remnant of a star that no longer undergoes fusion reactions.  
If it does not rotate, its total mass must be less than the Chandrasekhar limit. 
It resists gravitational collapse due to degenerate electron pressure.  
This leads to the standard equation of state in the basic mathematical model for a white dwarf, 
sometimes called the relativistically degenerate model.  
In this paper we consider a white dwarf that rotates 
about a fixed axis and thereby loses its spherical shape.  
Fixing its mass, we construct a connected set of steady-state rotating solutions.

The pressure $p$ of a white dwarf is given in terms of the density $\rho$  by the formula 
\eqn  \label{eqn of state}
p(\rho)  =  A\int_0^{\rho^{1/3}}  \frac{\sigma^4}{\sqrt{m^2+\sigma^2}}\ d\sigma,    \eeqn 
where $m$ is the mass of an electron and $A$ is a constant. 
The density $\rho$ evolves in time by the compressible Euler-Poisson equations (EP), 
subject to the internal forces of gravity due to the particles themselves.  
The speed $\omega(r)$ of rotation around the $x_3$-axis is allowed to depend on $r=r(x)=\sqrt{x_1^2+x_2^2}$.  
The inertial forces are entirely due to the rotation. In the region $\{x\in\real^3\ \Big|\ \rho(x)>0\}$ 
occupied by the star,  EP reduces to the equation (see Section 4 for details)
\eqn \label{NG} 
\frac 1{|x|}*\rho + \kappa^2\int_0^{r} s\omega^2(s)\,ds  
     -  h(\rho)   + \alpha =0,   \eeqn
where  $\kappa\omega(r)$ is the angular velocity, $\kappa$ is a constant measuring the intensity of rotation, 
 $h$ is the enthalpy defined by $h'(\rho)=\frac{p'(\rho)}{\rho}$ with $h(0)=0$, and $\alpha$ is a constant.  
We have normalized the physical constants. 
The density must vanish at the boundary of the star.

Non-rotating radial (spherically symmetric) white dwarfs were first analyzed by Chandrasekhar \cite{chandrasekhar1931maximum} (see also Chapter XI of \cite{chandrasekhar1939introduction}).  
He proved that there is a maximum mass $M_0$ for a white dwarf to exist. 
Auchmuty and Beals \cite{auchmuty1971variational} proved that for any $M<M_0$, there exists a rotating white dwarf of mass $M$ 
with compact support; it is obtained by minimizing the energy.  
Lieb and Yau \cite{lieb1987chandrasekhar} considered non-rotating white dwarfs as semi-classical limits  
of many quantum particles that are governed by a Schr\"odinger Hamiltonian.  

Our goal in this article is to prove that there is a global connected set of rotating solutions,  That is, 
it contains solutions which have arbitrarily large density somewhere or which have arbitrarily large support.  
They may rotate arbitrarily fast.  The conclusion is stated somewhat informally in the following theorem. 

\begin{theorem}     \label{thm: informal} 
Let $M$ be the mass of the non-rotating solution.  
Assume the pressure $p(\cdot)$ is given by \eqref{eqn of state} 
and the angular velocity $\omega(\cdot)$ satisfies .....
By a ``solution" of the problem for a rotating white dwarf, 
we mean a triple $(\rho, \kappa, \alpha)$, where  
$\rho$ is an axisymmetric function with mass $M$ that satisfies \eqref{NG} and 
$\kappa$ refers to the intensity of rotation speed. 
Then there exists a set $\K$ of solutions satisfying the following three properties. 
\begin{itemize}
\item $\K$ is a connected set in the function space $C_c^1(\real^3)\times \real\times\real$. 
\item $\K$ contains the non-rotating solution. 
\item either 
$$\sup \{ \rho(x)\ \Big| \  x\in\real^3, (\rho,\kappa,\alpha)\in\K \}=\infty$$

or 
$$\sup\{|x|\ \Big| \ \rho(x)>0, (\rho,\kappa,\alpha)\in\K \}=\infty.$$ 
\end{itemize}
The last statement means that 
 either the densities become pointwise unbounded or the supports become unbounded along $\K$.  
  \end{theorem}

In \cite{strauss2017steady} we constructed slowly rotating stars with fixed mass. 
In  \cite{strauss2019rapidly} we constructed a global connected set of slowly and rapidly rotating stars for 
a general class of equations of state.  However, the white dwarf case does not fall into this class.  
Keeping the mass constant is a key to our methodology, 
so that there is no loss or gain of particles when the star changes its rotation speed.   
Moreover, we permit a non-uniform angular velocity.  
 
A subtlety of the white dwarf case occurs in the proof that the total mass $M$ 
is a strictly monotone function both of the central density $\rho(0)$ and of the radius $R$ of the star 
in the non-rotating radial case.  We give a self-contained proof of this fact in Section 3.  
It is based on a fundamental lemma given in Section 2.  
The monotonicity is ultimately a consequence of the virial identity and the minimization of the energy.  
In a different context a weaker form of the monotonicity was proven in \cite{lieb1987chandrasekhar}.  
This monotonicity property of the mass is used in two crucial places in our proof in Section 4.    

In Section 4 we use the same basic method as in \cite{strauss2019rapidly}.  
That means we force the total mass $M$ to be fixed  and 
 introduce the constant $\alpha$ as a variable.   
 We get the support to be compact by artificially forcing the parameter $\alpha$ 
 to be sufficiently negative (see Lemma \ref{lem: support bound}).  
Then we begin the construction of rotating star solutions in the standard way by continuation 
from a non-rotating solution ($\kappa= 0$).    
Letting $\kappa$ increase, we continue the construction by applying a global implicit function theorem, 
which is based on the Leray-Schauder degree.  
Later on, in Theorem \ref{thm: main}, we obtain the whole global connected 
set $\K$ of solutions by allowing $\alpha$ to increase.  

The equation of state (1.1) for the white dwarf satisfies $p(\rho) = O(\rho^{4/3})$ as $\rho\to\infty$.  
However, the exact polytropic case $p=\rho^\gamma$ with $\gamma=\frac43$ was also excluded from \cite{strauss2019rapidly} 
because in that case the constant mass condition introduces a non-trivial nullspace of the linearized operator, 
which prevents the employment  of the implicit function theorem.  
Here we supplement our discussion of the white dwarf stars with a discussion of 
the polytropic case $p=\rho^\gamma$ with $\gamma=\frac43$.  
In that case we prove in Section 5 that that there is no slowly uniformly rotating solution at all 
with the given mass $M$.


\section{Preliminaries}
With the physical constants set to be 1, the equation of state is 
\eqn\label{eq: eos}
p(\rho)  =  \int_0^{\rho^{1/3}}  \frac{\sigma^4}{\sqrt{1+\sigma^2}}\ d\sigma.   \eeqn
We write $s=\rho$ for simplicity.  Note that explicit calculations yield 
\eqn  p'(s)=  \frac{s^{2/3}}{3\sqrt{1+s^{2/3}}} ,  \eeqn
\eqn h(s) = \sqrt{1+s^{2/3}} - 1,  \eeqn
\eqn  h'(s) = \frac{p'(s)}{s}  =  \frac1{3 s^{1/3}\ \sqrt{1+s^{2/3}}}.   \eeqn
Writing $t=h(s)\ge0$, we have 
\eqn  \label{def: h inv} h^{-1}(t) =  (2t+t^2)^{3/2},  \eeqn
\eqn   (h^{-1})'(t)  =  3(1+t)\sqrt{2t+t^2}.  \eeqn
 The key properties 
 (2.9)-(2.10) of our paper \cite{strauss2017steady}  are 
 \eqn  \lim_{s\to 0^+} s^{3-\gamma} p'''(s)  =  constant < 0  \eeqn
and 
\eqn  \lim_{s\to\infty}  s^{1-\gamma^*} p'(s) =constant >0.     \eeqn
They are true for $\gamma=5/3$ and  $\gamma^*=4/3$, respectively.    

In Section 3 we will have to study the equation 
\eqn   w_{rr} = g(w,r),  \eeqn 
where 
\eqn   \label{gwr}
g(w,r) = 4\pi rh^{-1}\left(\frac {w_+}r\right)  
=  4\pi r \left(2\frac {w_+}r + \frac{w_+^2}{r^2}\right)^{3/2}  \eeqn 
for $r>0$.  We need to understand how the solution depends on its data at $r=0$. 
This is described in the following lemma.  

\begin{lemma}   For $a>0$, denote by $w(r,a)$ the solution of 
$w_{rr} + g(w,r) =0$ with $w(0,a)=0$ and $w_r(0,a)=a$.  
Assume that for some $a_0>0$ and some $R>0$ 
we have $w(R,a_0)=0$ and $w(r,a_0)>0$ for all $0<r<R$.  
Then there exists $r_0\in(0,R)$ such that  
\eqn w_a>0 \text{ in } (0,r_0), \qquad   w_a<0 \text{ in } (r_0,R).  \eeqn
\end{lemma} 

\begin{proof}  The proof is closely related to Lemma 4.9 in \cite{strauss2017steady}.
We calculate 
\eqn   \label{gwgw}
\frac1{4\pi} (g-wg_w) 
=  \left(2\frac wr + \frac{w^2}{r^2}\right)^{1/2}   \left\{-w-\frac2r w^2\right\}  < 0,  \eeqn
which is (4.52) in \cite{strauss2017steady},  
and 
\eqn    \label{gr}
\frac1{4\pi} g_r  =  
\left(2\frac wr + \frac{w^2}{r^2}\right)^{1/2}   \left\{-\frac wr - \frac2{r^2} w^2\right\}  <0, \eeqn
which is (4.53) in \cite{strauss2017steady}. 
Furthermore, calculate 
\eqn     \label{rgr2g} 
 \frac1{4\pi} (rg_r+2g)  =  
\left(2\frac wr + \frac{w^2}{r^2}\right)^{1/2}  \{3w\}  > 0,  \eeqn 
which is weaker than (4.54) in \cite{strauss2017steady}.

Where convenient, we write $\frac{\pa}{\pa r}$ as $'$.  
We define the three auxiliary functions 
\begin{equation}
x(r;a) = rw'(r;a), ~y(r;a)=w'(r;a), ~z(r;a)=w_a(r;a).
\end{equation}
Their values at $r=0$ are 
\begin{equation}
x(0^+;a)=0, ~x'(0^+;a) = w'(0;a)-\lim_{r\to 0^+}rg(w,r)=a.
\end{equation}
\begin{equation}
y(0^+;a) = a,~ y'(0^+;a) = -\lim_{r\to 0^+}g(w,r)=0.
\end{equation}
\begin{equation}\label{phi int}
z(0^+;a) = 0,~z'(0^+;a) =1. 
\end{equation}
Now 
\eqn 
x'' = (rw')'' = rw'''+2w''= r(-g_r-g_ww')-2g=-rg_r-g_w x-2g.  \eeqn
So $x$ satisfies the equation
\begin{equation}
x'' + g_w x +rg_r +2 g=0.
\end{equation}
Similarly, 
\eqn
y'' + g_w y + g_r=0 ,  \quad  z'' + g_wz=0.    \eeqn
The derivatives of various Wronskians are 
\begin{equation}\label{wsk 1}
W(x,z)'=\begin{vmatrix} x &z \\ x'& z'\end{vmatrix}' = z(rg_r+2g).
\end{equation} 
\begin{equation}\label{wsk 2}
W(y,z)'=\begin{vmatrix} y &z \\ y' &z'\end{vmatrix}' = z g_r.
\end{equation} 
\begin{equation}\label{wsk 3}
W(w,z)'=\begin{vmatrix} w &z \\ w'& z'\end{vmatrix}' = z(g-wg_w).
\end{equation}

In the the rest of the proof we set $a$ equal to $a_0$ in all functions. 
Because $w>0$ and $w''=-g<0$ for $r\in (0,R)$, we see that $w$ is a positive concave function 
with a unique maximum and zero boundary values on $[0,R]$. 
By \eqref{phi int}, $z(r)>0$ for $r$ close to $0$. 

We claim that $z$ vanishes somewhere in $(0,R)$.   
On the contrary, suppose that $z(r)>0$ for all $r\in (0,R)$.  
Integrating \eqref{wsk 3} on $(0,R)$ and using the boundary conditions of $w$ and $z$, we have
\begin{equation}\label{eq: contra 1}
-w'(R)z(R) = \int_0^R z(g-wg_w)~dr<0.
\end{equation}
The inequality is a consequence of (2.10).  
However, since $w'(R)<0$ and $z(R)=z(R-)\ge 0$, 
the left side of \eqref{eq: contra 1} is non-negative.  
This contradiction shows that $z$ vanishes somewhere in the open interval.    

Let $r_0$ be the smallest value in $(0,R)$ for which $z(r_0)=0$. 
Integrating \eqref{wsk 1} on $(0,r_0)$, we find 
\begin{align}
x(r_0)z'(r_0)&= \int_0^{r_0}z(rg_r+2g)~dr>0
\end{align}
by \eqref{rgr2g} and the fact that $z(r)>0$ for $r\in (0,r_0)$. 
Since $z'(r_0)<0$, we deduce that $x(r_0)<0$, and hence $w'(r_0)<0$.   

Thus it suffices to show that $z(r)<0$ for all $r_0<r\le R$. 
Again supposing the contrary, let $r_1\in (r_0,R]$ be the first zero of $z$ strictly bigger than $r_0$.   
Integrating \eqref{wsk 2} on $(r_0,r_1)$ and recalling the definition $y=w'$, we obtain 
\begin{equation}\label{eq: contra 2}
w'(r_1)z'(r_1) - w'(r_0)z'(r_0) =y(r_1)z'(r_1) - y(r_0)z'(r_0) = \int_{r_0}^{r_1}z g_r~dr\ge 0.
\end{equation}
The last inequality follows from \eqref{gr} and the fact that $z(r)<0$ for $r\in (r_0,r_1)$.   
However, since $w$ is concave and $w'(r_0)<0$, it must also be the case that $w'(r_1)<0$. 
We also have $z'(r_0)<0$, and $z'(r_1)>0$.   
These conditions together imply that the left side of \eqref{eq: contra 2} is negative.  
This contradiction implies $z(r)<0$ for all $r_0<r\le R$.   
\end{proof}


\section{Monotonicity of the Mass}    \label{M' positive}

For a non-rotating (spherical) star, $\rho(0)$ is the density at its center.  Let $a = h(\rho(0))$. 
Denote the density of this star at any radius $r=|x|$ by $\rho(r;a)$ 
and denote the radius of the star by $R(a)$.   
Defining $u=h(\rho)$, it turns out that $\Delta u + 4\pi h^{-1}(u)=0$ for $r<R(a)$. 
The star's radius $R(a)$ is finite for all $a>0$, as is seen by applying the criterion 
$\int_0^1 h^{-1}(t) t^{-4}~dt = \infty$ of Theorem 1 in \cite{makino1984existence}.
The total mass of the star  is defined as 
\eqn  \label{mass}
M(a) = \int_{\real^3} \rho\ dx  =  4\pi\int_0^{R(a)}  \rho(r;a)\ r^2dr.
\eeqn
So $M'(a) = \int_{B(R(a))} \rho_a(x;a)\ dx$.  
Our goal is to prove the following lemma.  
\begin{lemma} 
$ M'(a)>0$ for all $a>0$.  
\end{lemma}

To this end we define the total energy  as 
\eqn   \E(\rho) = \int H(\rho)dx - D(\rho,\rho), 
\quad D(\rho,\rho) = \frac12 \iint \frac{\rho(x)\rho(y)}{|x-y|}dxdy.  \eeqn 
\begin{lemma}  
Any radial solution satisfies the virial identity  
\eqn    \E(\rho) = \int [4H(\rho)-3\rho h(\rho)]\ dx.   \eeqn
\end{lemma}
\begin{proof}  
We have $u=h(\rho)$ in $\Omega=:\{\rho>0\}$ 
and $\Delta u =  \frac1{r^2}(r^2u_r)_r= -4\pi \rho$ in $\real^3$.  
We consider $\rho$ to vanish outside $\Omega$.  
From the latter equation, we have 
\eqn
\int_{\real^3} |\nabla u|^2dx = 4\pi \int_{\real^3} \rho udx  
=  4\pi\int_{\real^3} \rho\left(\frac1{|\cdot|}*\rho\right)dx  =  2\pi D(\rho,\rho).   \eeqn
We therefore  have 
\eqn  \label{r3rho}
 \int_0^\infty r^3\rho h'(\rho) \rho_r dr  =  \int_0^\infty r^3\rho u_r dr  
 =  -\frac1{4\pi} \int_0^\infty r^2u_r (r^2u_r)_r \frac1rdr.   \eeqn 
Integrating by parts, the right side equals 
\eqn
\frac1{8\pi} \int_0^\infty (r^2u_r)^2 \frac1{r^2}dr  =  \frac1{8\pi} \int_0^\infty  u_r^2 r^2dr  
=  \frac1{32\pi^2} \int_{\real^3} |\nabla u|^2dx  =  \frac1{8\pi} D(\rho,\rho).    \eeqn
On the other hand, the left side of \eqref{r3rho} equals 
\eqn
\int_0^\infty  r^3 [\rho h(\rho) - H(\rho)]_rdr = -\int_0^\infty 3r^2 [\rho h(\rho) - H(\rho)]dr  
=  -\frac3{4\pi} \int_{\real^3} (\rho [h(\rho) - H(\rho)] dx.   \eeqn
Combining the last three equations, we have 
\eqn 
3 \int_{\real^3} [\rho (h(\rho) - H(\rho)]dx  =  D(\rho,\rho).  \eeqn
This proves (3.2).  
\end{proof}

{\it Proof of Lemma \ref{mass}.} 
The function  $u(r;a) = h(\rho(r;a))$  defined for $r\le R(a)$
 satisfies $\Delta u + 4\pi h^{-1}(u)=0$, $u(r;a)>0$ and $u_r(r;a)<0$  for $0< r<R(a)$,  
as well as the boundary conditions $u(R(a);a)=0, \ u(0,a)=a$.  
This function $u$  is extended to all of $\real^3$ by solving $\ \Delta u =  -4\pi h^{-1}(u_+)$ in $\real^3$.  
Thus $u$ us harmonic outside the star.   

Now we define $w=ru$.  
this change of variables gives us $\Delta w = g(w;r)$ for $0\le r \le R(a)$,  
where $g$ is defined in \eqref{gwr}.  
Also $w(0,a)=0, \ w_r(0,a)=a,\   w(R(a);a)=0$ and $w(r,a)>0$ for $0<r<R(a)$.    
Therefore Lemma 2.1 is applicable, so that $w_a$ strictly changes sign in the interval $(0,R(a))$.  
Now $w_a=ru_a=rh'(\rho)\rho_a$ and $h'>0$, so that $\rho_a$ 
also strictly changes sign in the interval $(0,R(a))$.  

From the definition of the energy $\E$, we have 
\begin{align}
\frac{d}{da} \E(\rho(\cdot,a))  &=  \int_{B(R(a))} h(\rho(x;a))\rho_a(x;a))  
-  \iint_{B(R(a))^2}  \frac{\rho_a(x;a)  \rho(x;a)} {|x-y|} dxdy  \\
&=  \int_{B(R(a))} \left\{ h(\rho(x;a)) - \left(\frac1{|\cdot|}*\rho(\cdot;a)\right)(x)\right\} \rho_a(x;a)\ dx  
\end{align}
since $\rho(R(a);a)=0$ and $H(0)=0$.  
By \ref{NG}  in Section 1, the factor in curly brackets is a constant $\alpha<0$, so that 
\eqn
\frac{d}{da} \E(\rho(\cdot,a))  = \alpha M'(a).  \eeqn
We will prove by contradiction that $M'(a)\ne0$.  

Now suppose that $M'(a)=0$ for some $a$.  Then $\frac{d}{da} \E(\rho(\cdot,a)) =0$. 
Using the virial identity, we therefore have 
\eqn 
0  =  \int  [h(\rho)-3\rho h'(\rho)]\ \rho_a \ dx  =  \int[(1+\rho^{2/3})^{-1/2}  -1]\ \rho_a\ dx.  
\eeqn 
The function $k(s) =1- (1+s^{2/3})^{-1/2}$ is positive and increasing for $s>0$,  
so that the radial function $r\to g(r)=:k(\rho(r))$ is positive and decreasing as a function of $r=|x|$ 
and it vanishes at $r=R(a)$.  
Now we have both $\int \rho_a\ dx=0$ and $\int g\ \rho_a\ dx=0$.  
This is impossible, due to the facts that $\rho_a$ strictly changes sign from positive to negative, 
while $g$ is positive and decreasing.  This contradiction means that $M'(a)\ne0$.  

Thus we have shown that $M(a)$ is either strictly increasing or strictly decreasing. 
We claim that  $M(a)\le C a^{3/4}$ for sufficiently small $a$. 
To prove this claim, we let $v(x;a)$ be the unique solution of 
\eqn\label{eq: rescale u}
\Delta v + (2v_++av_+^2)^{3/2}=0,~v(0)=1, ~v'(0)=0   \eeqn
 for $a\ge0$ and $|x|\ge0$. 
For $a>0$, a simple rescaling, 
using the formula for $h^{-1}$ given in \eqref{def: h inv} and the definition of $u$,   
shows that  $v(x;a) = \frac1a u\left(a^{-1/4}x;a\right)$,   
Now by  \cite{makino1984existence} the solution $v(x;0)$ has a unique zero  $R_0$.  
We obviously have $v'(R_0;0)<0$. 
By the continuous dependence of solution of the ODE on the parameter $a$, 
for arbitrarily small $\epsilon>0$ we have $v(R_0;a)<\epsilon$ and  
$v'(R_0;a)<v'(R_0;0)+\epsilon <0$, provided that $a$ is sufficiently small.  
Furthermore, $|v''(x;a)|$ is uniformly bounded for $|x|<R_0+1$ and $a$ small.   
Thus $v'(x;a)<v'(R_0;0)+2\epsilon$ for $R_0<|x|<R_0+\delta$ and some constant $\delta$.   
If $v(R_0;a)<0$, the zero of $v(x;a)$ occurs before $|x|$ reaches $R_0$.   
Otherwise $v(x;a)$ must cross zero before $|x|$ reaches $R_0+\delta$.   
Thus we have the following estimate on the radius of the star, 
which is the zero of $u(x;a)$: $R(a)\le a^{-1/4}(R_0+\delta)$ for small $a$.   
Because $u(x;a)$ and $\rho(r;a)$ are radially decreasing, we have  $\rho(r;a) \le h^{-1}(a)$ and 
$M(a)\le Ch^{-1}(a) [R(a)]^3 \le Ca^{3/2} a^{-3/4}=Ca^{3/4}$ for sufficiently small $a$.  
This proves the claim. 
Now if we assume by contradiction that $M(\cdot)$ is decreasing, then let $0<\ep<a$.  It follows that $0\le M(a) \le M(\ep) \le C\ep^{3/4}$ for small $\ep$.  
Hence $M(\cdot)$ cannot be decreasing.  Therefore $M'>0$.  
\qed

\section{Existence of Rotating White Dwarf Solutions}

We first describe how EP  reduces to \eqref{NG}.  
 The compressible Euler-Poisson equations (EP) are 
 \begin{equation}\label{eq: full Euler-Poisson}
\begin{cases}
\rho _t + \nabla\cdot (\rho v)=0, \\
(\rho v)_t + \nabla\cdot(\rho v\otimes v) + \nabla p = \rho \nabla U, \\
U(x,t) = \int_{\real^3}\frac{\rho(x',t)}{|x-x'|}~dx'.
\end{cases}
\end{equation}
The first two equations hold where $\rho>0$, and the last equation defines $U$ on the entire $\real^3$.
The equation of state $p = p(\rho)$ given by \eqref{eq: eos} closes the system. 
To model a rotating star, one looks for a steady axisymmetric rotating solution 
to \eqref{eq: full Euler-Poisson}.  
That is,  we assume $\rho$ is symmetric about the $x_3$-axis  
and $v=\ka\,\omega(r)(-x_2,x_1,0) $, where $r=r(x)=\sqrt{x_1^1+x_2^2}$, 
as distinguished from the $r$ in Section 3,  
with a prescribed function $\omega(r)$. 
With such specifications, the first equation in \eqref{eq: full Euler-Poisson} 
concerning mass conservation 
is identically satisfied. The second equation in \eqref{eq: full Euler-Poisson} 
concerning momentum conservation simplifies to
\begin{equation}\label{eq: pre Euler-Poisson}
-\rho\, \ka^2\,r\omega^2(r) e_r + \nabla p 
= \rho\,\nabla \left(\frac{1}{|\cdot|}*\rho\right), \qquad e_r=\frac{1}{r(x)}(x_1,x_2,0). 
\end{equation}
The first term in \eqref{eq: pre Euler-Poisson} can be written as
$-\rho\nabla \left(\int_0^r \omega^2(s)s~ds\right).  $
Introducing the {specific enthalpy} $h$ as above, 
\eqref{eq: pre Euler-Poisson} becomes 
\eqn \label{eq: Euler-Poisson vector}
\nabla \left(\frac{1}{|\cdot|}*\rho +\ka^2\int_0^r \omega^2(s)s~ds - h(\rho) \right)=0.  \eeqn

With the key difficulty about the mass  function $M(a)$ having been resolved in Section \ref{M' positive},  
we will be able to prove Theorem \ref{thm: informal}.  
In order to formulate the result precisely, let us put the following conditions on the 
rotation profile $\omega(s)$:
\eqn\label{cond: omega 1}
s\omega^2(s)\in L^1(0,\infty),\quad \omega^2(s)\text{ is not compactly supported},
\eeqn
\eqn\label{cond: omega 2}
\lim_{r(x)\to\infty}r(x)(\sup_xj-j(x))=0, 
\eeqn
where 
\eqn
j(x) = \int_0^{r(x)}s\omega^2(s)~ds.
\eeqn
Let $\rho_0(x)$ be the unique non-rotating ($\kappa=0$) solution with mass 
 $M=\int\rho_0(x)~dx$.   We define the pair  
$\F(\rho,\kappa,\alpha) = (\F_1(\rho,\kappa, \alpha) , \F_2(\rho))$, where 
\eqn
\F_1(\rho,\kappa, \alpha)  =  \rho(\cdot)  -   h^{-1}\left(
\left[ \frac{1}{ |\cdot |} * \rho(\cdot) + \kappa^2j(\cdot) + \alpha \right]_+\right),  
\eeqn  and 
\eqn
\F_2(\rho)  =  \int_{\real^3} \rho(x)\ dx - M.  
\eeqn

As usual, a solution to $\F(\rho,\kappa,\alpha)=0$ 
with $\rho\in C_{loc}(\real^3)\cap L^1(\real^3)$ will give rise to a steady solution of the 
Euler-Poisson equations with rotation profile $\kappa\omega(s)$, and mass $M$.  
In particular, $\F(\rho_0,0,\alpha_0)=0$.  
We have the following main theorem, more precise than Theorem \ref{thm: informal}.
\begin{theorem}\label{thm: main}
For given $\omega(s)$ satisfying the above assumptions, and given non-rotating white 
dwarf solution $\rho_0$, there exists a connected set $\K$ in $C_c^1(\real^3)\times \real\times \real$ 
such that 
\begin{enumerate}
\item $\F(\rho,\kappa,\alpha)=0$ for all $(\rho,\kappa,\alpha)\in\K$.  
In other words, $\K$ is a set of rotating white dwarf solutions.
\item $(\rho_0,0,\alpha_0)\in \K$.
\item Either 
$$\sup\{\|\rho\|_\infty~|~(\rho,\kappa,\alpha)\in\K\}=\infty$$
or
$$\sup\{|x|~|~\rho(x)>0,~(\rho,\kappa,\alpha)\in\K\}=\infty.$$
\end{enumerate}
\end{theorem}
This means that either the densities become unbounded or the supports of the stars 
become unbounded. 
The proof of Theorem \ref{thm: main} is basically parallel to the argument 
in \cite{strauss2019rapidly} now that we have proven $M'(a)\ne 0$ in Section \ref{M' positive}. 
For completeness, we provide a sketch of the complete argument below.   
We refer to \cite{strauss2019rapidly} for more details.

For fixed constants $s>3$, 
let us define the weighted space 
$$ 
\C_s = \left\{  f:\real^3\to \real\ \Big|\ f \text{ is continuous, axisymmetric, even in }x_3, 
\text{ and } \|f\|_s <\infty\right\},  $$ 
where 
$$
\|f\|_s =: \sup_{x\in\real^3}\langle x\rangle^s|f(x)| <\infty.$$ 
Also define for $N>0$,
\eqn  \label{Oh}
\O_N = \left\{ (\rho,\kappa,\alpha)\in \C_s\times\real^2\ \Big|\  \alpha +  \kappa^2 \sup_x j(x) < - \frac1N \right\}. 
\eeqn 

We begin by showing an elementary support estimate for the nonlinear part of $\F_1$ on $\O_N$.
\begin{lemma} \label{lem: support bound}
There exists a constant $C_0$ such that 
for all $ (\rho,\kappa,\alpha)\in \O_N$ the expression 
$\left[ \frac{1}{ |\cdot |} * \rho(\cdot) + \kappa^2j(\cdot) + \alpha \right]_+$ 
is supported in the ball $\{x\in\real^3:\ |x| \le C_0 N\|\rho\|_s \}$. 
\end{lemma}
\begin{proof} 
First we note that $\left| \frac{1}{ |\cdot |} * \rho(\cdot) (x)\right|\le C_0 \|\rho\|_s \frac1{\langle x\rangle}$  
because $s>3$.    Hence for $|x|>C_0N\|\rho\|_s$,
$$
\left[ \frac{1}{ |\cdot |} * \rho(\cdot) (x)+ \kappa^2j(x) + \alpha \right] 
\le C_0 \|\rho\|_s \frac1{\langle x\rangle} - \frac1N < 0  $$
since $ (\rho,\kappa,\alpha)\in \O_N$.  Therefore its positive part vanishes for such $x$.  
\end{proof}

We see from this lemma that $\F_1$ differs from $\rho$ only by a perturbation on a compact set.   
Using this observation and the smoothing effect of $\Delta^{-1}$, it is easy to obtain 
\begin{lemma}\label{lem: compactness}
$\F$ maps $\O_N$ to $\C_s\times \real$. It is $C^1$ Frech\'et differentiable, 
where $\frac{\partial\F}{\partial(\rho,\alpha)}(\rho,\kappa,\alpha)$ is Fredholm of index zero.   
The nonlinear part of $\F_1$ (i.e. $\F_1-\rho$) is compact from $\O_N$ to $\C_s$.
\end{lemma}
\begin{proof} 
By Lemma \ref{lem: support bound},  if $(\rho,\kappa,\alpha)$ is bounded, the support of 
$\left[ \frac{1}{ |\cdot |} * \rho(\cdot) + \kappa^2j(\cdot) + \alpha \right]_+$ is contained in some 
ball $B_R$. The map is obviously compact from $\O_N$ to $C^0(\overline{B_R})$.   
Using again the trivial bound $\|u\|_{\C^s}\le \langle R \rangle^s \|u\|_{C^0(\overline{B_R})}$ 
for $u\in \C_s$ supported in $B_R$, we obtain the  compactness of this mapping into $\C_s$.
\end{proof}

\begin{lemma} \label{lem: iso}
$\frac{\partial \F}{\partial(\rho,\alpha)} (\rho_0,0,\alpha_0): \C_s\times\real\to\C_s\times \real$ is an isomorphism.  
\end{lemma} 
\begin{proof} 
This lemma is the first place where the crucial condition $M'(a)\ne 0$ proven in Section \ref{M' positive} will be used. 
Let $(\delta\rho,\delta\kappa)$ belong to the nullspace of  $\frac{\partial \F}{\partial(\rho,\alpha)} (\rho_0,0,\alpha_0)$. 
Let $w = \frac1{|\cdot|} *\delta\rho + \alpha_0$.  
As shown in Lemma 4.3 of \cite{strauss2019rapidly}, $w$ is radial.  
Indeed, that argument  shows that $w$ must  be a radial solution of the 
boundary value problem
\eqn\label{eq: trivial kernel ode}
\Delta w + 4\pi \left[\left(h^{-1}\right)'(u_0)\right] w =0,\quad w'(0)=w'(R_0)=0
\eeqn
in the ball $B_{R_0}$, where $B_{R_0}$ is the support of $\rho_0$, and $u_0=h(\rho_0)$.  
Being an ODE, \eqref{eq: trivial kernel ode} can have an at most a one-dimensional solution space.  
On the other hand, we recall the definition for any $a>0$ that  $u(r;a)$ solves 
\eqn
\Delta u + 4\pi h^{-1} (u) =0,\quad u'(0)=0,~u(0;a)=a. 
\eeqn
Denoting  $u_a=\partial_au(r;u_0(0))$, we obviously have 
\eqn\label{eq: trivial kernel ode1}
\Delta u_a+ 4\pi \left[\left(h^{-1}\right)'(u_0)\right] u_a =0,\quad u_a'(0)=0.
\eeqn
Comparing \eqref{eq: trivial kernel ode} with \eqref{eq: trivial kernel ode1}, we see that $w=Cu_a$ for some constant $C$. Integrating \eqref{eq: trivial kernel ode}, we also have
\eqn
\int_{B_{R_0}}\left[\left(h^{-1}\right)'(u_0)\right] Cu_a~dx=0.
\eeqn
On the other hand, taking account of $\rho = h^{-1}(u)$ and \eqref{mass},  we see that  
\eqn
\int_{B_{R_0}}\left[\left(h^{-1}\right)'(u_0)\right] u_a~dx=\frac{d}{da}\bigg|_{a=u_0(0)}\int_{u(x;a)>0} h^{-1}(u(x;a))~dx=M'(u_0(0))\ne 0.
\eeqn
There is no boundary term because $h^{-1}(0)=0$. 
The last two equations imply that $C=0$, so that $w=0$.  
This implies that the kernel of $\frac{\partial \F}{\partial(\rho,\alpha)} (\rho_0,0,\alpha_0)$ is trivial, 
which is the key ingredient of the operator  being an isomorphism. 
\end{proof}

\begin{proof}  [Proof of  Theorem \ref{thm: main}]
With the suitable compactness properties given by Lemma \ref{lem: compactness} and local solvability given by Lemma \ref{lem: iso}, one is in a position to apply a global implicit function theorem of Rabinowitz (see Theorem 3.2 in \cite{rabinowitz1971some}, Theorem II.6.1 of \cite{kielhofer2006bifurcation},   
or \cite{alexander1976implicit}). The result is a connected set $\K_N\subset \O_N$ of solutions to $\F=0$ for which at least one of the following three properties holds:
\begin{enumerate}
\item $\K_N\setminus \{(\rho_0,0,\alpha_0)\}$ is connected.
\item $\K_N$ is unbounded, i.e. $$\sup_{\K_N} \ (\|\rho\|_s + |\kappa| + |\alpha|) =\infty.$$
\item $\K_N$ approaches the boundary of $\O_N$, i.e. 
$$\inf_{\K_N} \ \left |\kappa^2 \sup_x j(x) + \alpha + \frac1N \right | = 0.$$
\end{enumerate}

The first alternative (the 'loop') can be eliminated by observing that, 
since $\K_N$ is even in $\kappa$, 
if $\K_N\setminus \{(\rho_0,0,\alpha_0)\}$ were connected, it must contain a different non-rotating solution $(\rho_1,0,\alpha_1)\ne (\rho_0,0,\alpha_0)$.   
As in Lemma 5.1 of \cite{strauss2019rapidly},  it must be 
 a radial non-rotating white-dwarf solution with a different center density $\rho_1(0)\ne\rho_0(0)$ 
 but with the same total mass $\int_{\real^3}\rho_1(x)~dx=\int_{\real^3}\rho_0(x)~dx$. 
 This contradicts the strict monotonicity of $M(a)$ established in Section \ref{M' positive}.

The sets $\K_N$ are nested, so  their union $\K = \cup_{N=1}^\infty \K_N$ 
is also connected.  Therefore one of the following statements is true:
\begin{enumerate}[(a)]
\item $\sup_{\K} \ (\|\rho\|_s + |\kappa| + |\alpha|) = \infty$.
\item $\inf_\K \ |\kappa^2 \sup_x j(x) + \alpha| = 0$.  
\end{enumerate}
We suppose that both $\sup_\K \sup_{x\in\real^3} \rho(x)<\infty$ 
and $R_* =: \sup_\K  \sup_{ \rho(x)>0} \ |x| \  <\infty$, and will derive a contradiction.   

We will first prove that (a) is true.  
On the contrary, suppose that (a) is false.  Then (b) must be true and 
$\sup_{\K} \ (\|\rho\|_s + |\kappa| + |\alpha|) < \infty$. 
Since $|x-y|\le |x|+R_*$ for all $y$ in the support of $\rho$,   
we have $$\left(\frac1{|\cdot|}*\rho\right) (x) = \int \frac1{|x-y|} \rho(y) dy  \ge  \frac M{|x|+R_*}.$$ 
We may now write
\eqn\label{est: lower bound 1}
\frac1{|\cdot|}*\rho(x)  + \kappa^2j(x) + \alpha \ge \frac{M}{|x|+R_*}  
-\kappa^2 (\sup j - j(x))+(\kappa^2 \sup j + \alpha).
\eeqn
Let $\kappa_0=\sup_{\K}|\kappa|$.   
Considering a point $x$ in the plane $\{x_3=0\}$, we have  $|x|=r(x)$.   
By \eqref{cond: omega 2}, $\sup j-j(x) = o\left(\frac1{|x|}\right)$ as $|x|\to\infty$.
Thus by \eqref{est: lower bound 1}, 
\eqn\label{est: lower bound 2}
\frac1{|\cdot|}*\rho(x)  + \kappa^2j(x) + \alpha \ge\frac{M}{|x|+R_*}   
 - o\left(\frac{\kappa_0^2}{|x|}\right)+(\kappa^2 \sup_xj(x) + \alpha).
\eeqn
Choosing $|x|>R_*$ sufficiently large, we can make the sum of the first two terms on the right side 
of \eqref{est: lower bound 2} positive.   
Then because of (b), there exists a solution $(\rho,\kappa,\alpha)\in \K$ such that 
the right side of \eqref{est: lower bound 2} is positive.  
Hence, due to $\F_1(\rho, \kappa, \alpha) = 0$, we have $\rho(x)>0$.  
This contradicts the assumption that the support of $\rho$ is bounded by $R_*$.  
Thus (a) must be true.  

Since we have assumed that $\rho$ is pointwise bounded and its support is also bounded all along $\K$, 
it follows that $\rho$ is also bounded in the weighted space $\C_s$.  
Because of (a), it must be the case that $|\kappa| + |\alpha|$ is unbounded.  
From the definition of $\O_N$, we know that $\alpha<0$.  
In case $\kappa$ were bounded, it would have to be the case that $\alpha \to -\infty$ along a sequence.  
Then the equation $\F_1=0$ would imply that $\rho\equiv0$, which contradicts the mass constraint.  

So it follows that $\kappa_n \to \infty$  for some sequence 
$(\rho_n,\kappa_n,\alpha_n)\in \K$ with $\alpha_n<0$.   
For each $n$, let us choose any point $x_n$ such that $\rho_n(x_n)>0$.   
By \eqref{cond: omega 1}, we may also choose a point $y_0$ such that $r(y_0)>R_*$ and $j(y_0)>j(R_*)$.  
Since  $\rho_n(y_0)=0$ and $\rho_n(x_n)>0$, we have 
$$
0 \ge \left[\frac{1}{ |\cdot |} * \rho_n(\cdot) + \kappa_n^2j(\cdot) + \alpha_n\right]  (y_0) 
\ge \left[ \frac{1}{ |\cdot |} * \rho_n(\cdot) + \kappa_n^2j(\cdot) + \alpha_n\right] \bigg|_{x_n}^{y_0}.   $$   
On the right side, the $\alpha_n$ cancels.  
Due to our assumption that 
the values of $\rho_n$ and the supports of $\rho_n$ are uniformly bounded, 
we deduce that 
$$ 
0 \ge \kappa_n^2 [j(r(y_0))-j(r(x_n))] - C,$$
where $C$ is a fixed constant .  
Thus $j(r(x_n)) \to j(r(y_0))$ since $\kappa_n\to\infty$. 
But $r(x_n)\le R_*<r(y_0)$ and $j$ is an increasing function of $r$, so that 
$j(r(x_n)) \le j(R_*) < j(r(y_0)).$  
This is the desired contradiction.  

Finally, we remark on why $\K$ is also connected in $C_c^1(\real^3)\times \real^2$.   
In fact, we know that for each $N$ the set $\K_N$ is connected in $\C_s\times \real^2$. 
We also know from Lemma \ref{lem: support bound} that all the solutions in $\K_N\subset \O_N$ 
have a uniform bound on their supports,  This bound may depend  on $N$. 
The regularizing effect of $\Delta^{-1}$ then implies that $\K_N$ is connected in $C^1(F_N)\times\real^2$ 
for a suitable compact set $F_N\subset \real^3$.   
Thus $\K_N$ is connected in $C^1_c(\real^3)\times\real^2$ under the usual direct limit topology. 
Because $\K$ is a nested union of $\K_N$, it too is connected in $C^1_c(\real^3)\times\real^2$.

\end{proof}


\section{Pure $4/3$ Power under Uniform Rotation}
In this section, we briefly study the Euler-Poisson equation under the pure power equation of state $p=\rho^{4/3}$ and constant angular velocity profile. 
Analogously to the white dwarf case, we define $\F = \F(\rho,\kappa,\alpha) = (\F_1,\F_2)$ by 
\begin{align}
\F_1(\rho,\kappa,\alpha) &= \rho(x) -\left[\frac{1}{|\cdot|}*\rho(x)-\frac{1}{|\cdot|}*\rho(0)+\kappa r^2(x)+\alpha\right]_+^3,\label{eq: wd def F1}\\
\F_2(\rho,\kappa,\alpha) &= \int_{B_1}\rho(x)~dx-M,
\end{align}
and solve for $\F(\rho,\kappa,\alpha)=(0,0)$. The cubic function in \eqref{eq: wd def F1} corresponds to the pure $\frac43$ power in the equation of state. As before, the radial non-rotating solution $\rho_0 = u_0^3$ satisfies the equivalent equation 
\eqn\label{eq: 4/3 diff}
\Delta u_0 + 4\pi u_0^3 =0
\eeqn
on its support, which we may take to be the unit ball $B_1$ without loss of generality.   
Let $\alpha_0 = u_0(0)$, and $M=\int_{B_1}\rho_0(x)~dx$.   
We readily check that $\F(\rho_0,0,\alpha_0)=(0,0)$. By the scaling symmetry of \eqref{eq: 4/3 diff}, 
we easily see that for any $\alpha>0$ and 
$$\rho^\alpha(x) = \left(\frac\alpha{\alpha_0}\right)^3\rho_0\left(\frac\alpha{\alpha_0} x\right),  $$ 
we have $\F(\rho^\alpha, 0,\alpha) =(0,M)$.
This $\rho^\alpha$ has the same mass $M$ for all $\alpha$.   
 
Let $X = \C_{sym}(\overline{B_2})$ be  defined to have the same symmetry properties 
as $\C_s$ but  only defined on $B_2$.   
We will show that the linear operator 
$\frac{\partial \F}{\partial (\rho,\kappa)}(\rho_0,0,\alpha_0)):X\times \real\to X\times \real$ is bijective.   
Once this is proven,  the implicit function theorem implies that $(\rho,\kappa)$ is locally 
 uniquely determined locally by $\alpha$.  
Therefore the trivial solutions $(\rho^\alpha,0,\alpha)$ defined above are the unique local solutions 
and they are non-rotating.   
We thus obtain the following curious conclusion. 

\begin{prop} Assuming the equation of state $p=\rho^{4/3}$ and the 
uniform rotation profile $\omega\equiv\kappa$,   
there are no solutions close to $\rho_0$ with the same total mass as $\rho_0$ that are  slowly rotating.
\end{prop}

\begin{proof}
We just need to prove the bijectivity.     We compute the derivative of $\F$ as follows, recalling that 
$u_0=\rho_0^{1/3} = \left[\frac{1}{|\cdot|}*\rho_0(x)-\frac{1}{|\cdot|}*\rho_0(0)+\alpha_0\right]_+$.   
\begin{align}\label{eq: ker 1}
&~\frac{\partial \F_1}{\partial (\rho,\kappa)}\bigg|_{(\rho,\kappa,\alpha)=(\rho_0,0,\alpha_0)}(\delta\rho,\delta\kappa)\notag\\
=&~\delta\rho - 3u_0^2\left[\frac1{|\cdot|}*\delta\rho(x)-\frac1{|\cdot|}*\delta\rho(0)+\delta\kappa r^2(x)\right],
\end{align}
\eqn\label{eq: ker 2}
\frac{\partial \F_2}{\partial (\rho,\kappa)}\bigg|_{(\rho,\kappa,\alpha)}(\delta\rho,\delta\kappa) = \int_{B_2}\delta\rho(x)~dx.
\eeqn
This derivative is a compact perturbation of the identity  and thus is Fredholm of index zero.   
Hence we merely need to show it is injective.   
To that end, let us assume that \eqref{eq: ker 1} and \eqref{eq: ker 2} both vanish.  
Denoting 
\eqn
\varphi(x) = \frac{1}{|\cdot|}*\delta\rho(x)-\frac{1}{|\cdot|}*\delta\rho(0)+\delta\kappa r^2(x),   
\eeqn
we then have  
\eqn\label{eq: 4}
\Delta \varphi = -4\pi \delta \rho+ 4\delta \kappa = -12\pi u_0^2\varphi +4\delta\kappa,
\eeqn 
and 
\eqn\label{eq: 5}
\int_{B_1}u_0^2\varphi ~dx=0.
\eeqn
We project \eqref{eq: 4} onto the radial component (integrating against 1 on $\mathbb{S}^2$), 
where $\varphi_{00}$ denotes the radial component of $\varphi$,  to obtain 
\eqn\label{eq: 6}
\Delta \varphi_{00} =  -12\pi u_0^2\varphi_{00} +4\delta\kappa,
\eeqn
while \eqref{eq: 5} naturally selects the radial component so that 
\eqn\label{eq: 7}
\int_{B_1}u_0^2\varphi_{00} ~dx=0.
\eeqn
If $\delta\kappa\ne 0$, we can divide \eqref{eq: 6} by it, and without loss of generality, we may assume $\delta\kappa=1$. Integrating \eqref{eq: 6} on $B_1$ and using \eqref{eq: 7}, we get
\eqn
4\pi \varphi_{00}'(1) = 4\frac{4\pi}{3},
\eeqn
\eqn
 \varphi_{00}'(1) =\frac43.
\eeqn
Then the function $u(|x|) = \varphi_{00}(|x|)-\frac23 |x|^2$ satisfies  
\eqn\label{eq: 10}
\Delta u + 12\pi u_0^2u = -8\pi u_0^2|x|^2  
\eeqn
and 
\eqn\label{eq: bdy u}
u'(1)=0.
\eeqn
Referring to the proofs of Lemma 4.3 and Lemma 4.7 in \cite{strauss2017steady} 
in the case that $\gamma=\frac43$, the radial function 
\eqn\label{eq: def va}
v(|x|) = \frac{\partial}{\partial \alpha}(\rho^\alpha(|x|))^{1/3} \bigg|_{\alpha=\alpha_0}= u_0(|x|)+ru_0'(|x|)
\eeqn
satisfies on $B_1$
\eqn\label{eq: va diff}
\Delta v + 12\pi u_0^2 v =0,
\eeqn
and
\eqn\label{eq: bdy va}
v'(1)=0,
\eeqn
\eqn\label{eq: 14}
\int_{B_1} u_0^2v~dx=0.
\eeqn
In fact, \eqref{eq: va diff} is a special case of (4.27) in \cite{strauss2017steady} 
(where $h^{-1}(s) = s^3$). (4.28) in \cite{strauss2017steady} shows the left hand side 
of \eqref{eq: bdy va} and that of \eqref{eq: 14} are the same. (4.45) in \cite{strauss2017steady} 
implies \eqref{eq: bdy va}, and finally \eqref{eq: def va} follows from (4.44) 
in \cite{strauss2017steady} (ignoring an irrelevant constant multiple).
We multiply \eqref{eq: 10} by $v$, multiply \eqref{eq: va diff} by $u$, and take the difference,  
obtaining 
\eqn\label{eq: pre green}
v\Delta u - u\Delta v = -8\pi u_0^2 |x|^2 v.
\eeqn
Integrating \eqref{eq: pre green} over $B_1$, using Green's identity and the boundary conditions \eqref{eq: bdy u} and \eqref{eq: bdy va}, we get
\eqn\label{eq: 15}
\int_{B_1}u_0^2v |x|^2~dx=0.
\eeqn

But notice that 
\eqref{eq: 14} and \eqref{eq: 15} contradict each other!  
Indeed, $v' = 2u_0' +ru_0''=2u_0'-4\pi |x|u_0^3<0$ for $|x|<1$. 
It follows from \eqref{eq: 14} that $u_0^2 v$ is positive near $0$ and negative near $\partial B_1$, 
and only switches sign once. Therefore \eqref{eq: 15} cannot hold. 
This contradiction implies that  $\delta\kappa=0$.   
Then the same argument as in the proof of Lemma 4.3 in \cite{strauss2017steady} 
shows that $\delta \rho=0$. 
\end{proof}

\noindent{\bf Acknowledgments.}  
YW is supported by NSF Grant  DMS-1714343.

\bibliographystyle{acm}
\bibliography{rotstarbiblio}

\end{document}